\documentclass{amsart}

\usepackage[margin=1.3in]{geometry}
\usepackage{amssymb,enumerate,amsmath,amsxtra,geometry}

\newtheorem*{Lem}{Lemma}

\newtheorem*{Thm}{Theorem}

\allowdisplaybreaks

\numberwithin{equation}{subsection}

\theoremstyle{remark}
\newtheorem*{Rmk}{Remark}

\newcommand{\vform}{\langle\mspace{7mu},\mspace{7mu}\rangle}
\newcommand{\hform}{ (\mspace{7mu},\mspace{6mu})}

\newcommand{\tp}{{}^{\top}}

\newcommand{\pic}{\pi^{\vee}}

\newcommand{\ui}{{}^{\iota}}

\newcommand{\ta}{\theta}
\newcommand{\la}{\lambda}
\newcommand{\ve}{\varepsilon}
\newcommand{\fac}{\Bar{F}}
\newcommand{\Ac}{\Bar{A}}

\newcommand{\ut}{{}^{\theta}}
\newcommand{\utt}{{}^{\Bar{\theta}}}

\newcommand{\ms}{\mathcal{S}}
\newcommand{\msb}{\Bar{\mathcal{S}}}
\newcommand{\nr}{{\rm N}_A}
\newcommand{\nrb}{{\rm N}_{\Bar{A}}}
\newcommand{\ia}{\iota}
\newcommand{\uia}{{}^{\iota}}

\newcommand{\vad}{V_a^{\vee}}

\newcommand{\mcp}{\mathcal{P}}

\newcommand{\lag}{\mathfrak{g}}

\newcommand{\ml}{\mathcal{L}}

\newcommand{\od}{\mathfrak{o}_D}

\newcommand{\fo}{\mathfrak{o}_F}

\begin{document}

\thanks{2010 {\em Mathematics Subject Classification.} 16W10, 22E50, 15A24.}
\thanks{Vinroot was supported in part by a grant from the Simons Foundation, Award \#280496}

\title{A note on dual modules and the transpose}
\makeatletter
\def\author@andify{%
\nxandlist {\unskip ,\penalty-1 \space\ignorespaces}%
   {\unskip {}  \@@and~}%
   {\unskip \penalty-2 \space  \@@and~}%
   }
 \makeatother

\author[T. Madsen]{Thomas Madsen}
\address{Dept. of Mathematics and Statistics, Youngstown State University, Youngstown, OH 44555.} 
\email{tlmadsen@ysu.edu}
\author[A. Roche]{Alan Roche}
\address{Dept. of Mathematics, University of Oklahoma, Norman, OK 73019-3103.}
\email{aroche@math.ou.edu}
\author[C. R. Vinroot]{C.~Ryan Vinroot}
\address{Dept. of Mathematics, College of William and Mary, P.O.~Box 8795,  Williamsburg, VA 23187-8795.}
\email{vinroot@math.wm.edu}


\begin{abstract}
It is a classical result in matrix algebra that any square matrix over a field can be conjugated to its transpose by a symmetric matrix. For $F$ a non-Archimedean local field, Tupan used this to give an elementary proof that  transpose inverse takes each irreducible smooth 
repesentation of ${\rm GL}_n(F)$ to its dual.  
We re-prove the matrix result and related observations using module-theoretic arguments. In addition, we write down a generalization that applies to central simple algebras with an involution of the first kind. We use this generalization to extend Tupan's method of argument to 
${\rm GL}_n(D)$ for $D$ a quaternion division algebra over $F$. 
\end{abstract}

\maketitle 

\section*{Introduction}
Let $F$ be a field and let $a$ be a square matrix over $F$.
Writing $\top$ for transpose, it is well known that there is an invertible matrix $g$ over $F$ 
such that $g a g^{-1} = \tp a$ and $\tp g = g$
(see, for example, \cite[2.6]{Kap} or \cite{TTZ}). 

Our first object is to extend this classical matrix statement. 
We do so by replacing the pair $({\rm M}_n(F), \top)$ by 
$(A, \ta)$ where $A$ is a central simple algebra over $F$ and $\ta$ is an involution on $A$ of the first kind. 
By definition, the map $\ta$ is $F$-linear, reverses multiplication and satisfies  
 $\ut (\ut a) = a$ for all $a \in A$. 
For $\fac$ an algebraic closure of $F$, we have $\Ac = A \otimes_F \fac \cong {\rm M}_n(\fac)$
for $n^2 = \dim_F A$. 
The extended map $\Bar{\ta} = \theta \otimes 1_{\fac}$ is then
an involution of the first kind on $\Ac$. For any $b \in {\rm GL}_n(\fac)$, we write ${\rm Int} \,b$ for the 
inner automorphism of ${\rm M}_n(\fac)$ given by conjugation by $b$. 
By a standard argument, any isomorphism 
$\Ac \cong {\rm M}_n(\fac)$ takes $\Bar{\ta}$ to a composition ${\rm Int}\, b \circ \top$
where $\tp b = \ve \, b$ for $\ve = \pm 1$.  The sign $\ve$ is independent of  the choice of $b$ and the choice of isomorphism $\Ac \cong {\rm M}_n(\fac)$. Accordingly, we write $\ve  = \ve(\ta)$. 
Our extension of the classical matrix result is as follows: 
\[
\text{for any $a \in A$, there is a $g \in A^\times$ such that $gag^{-1} = \ut a$ and 
$\ut g = \ve(\ta) \,g$.}    \tag{$\star$}
\]

Suppose now that $F$ is a non-Archimedean local field. Tupan used the classical matrix result 
and some $p$-adic topology to give an elementary proof that transpose inverse takes each 
irreducible smooth representation of ${\rm GL}_n(F)$ to its dual.
This was first established by Gelfand and Kazhdan by a geometric method \cite{GK}. Raghuram extended Gelfand-Kazhdan's method  to the group ${\rm GL}_n(D)$ where $D$ is a quaternion 
division algebra over $F$ \cite{Ragh}.
Using $(\star)$, it is a simple matter to extend Tupan's arguments to ${\rm GL}_n(D)$. We record the details in \S \ref{application} below. 

In the final section of the paper, we re-prove the classical matrix statement and related observations from \cite{TTZ}. In place of matrix computations, we use some standard facts about finitely generated torsion modules over PIDs. 
While we certainly do not match the brevity or efficiency of the arguments in \cite{TTZ}, there may be some merit in recording our conceptual approach.

\section{Conjugacy and Involutions}
Let $F$ be a field and let $A$ be a central simple $F$-algebra. Thus the  $F$-algebra $A$ 
admits no proper nonzero two-sided ideals and has center $F$. 
For us also, $A$ always has finite dimension as a vector space over $F$. 
Let $\ta$ be an involution on $A$. That is, 
\begin{enumerate}[a)]
\item
$\ta:A \to A$ is $F$-linear,
\item
$\ut(ab) = \ut b \, \ut a$ for all $a, b \in A$, 
\item
$\theta \circ \theta = 1_A$, the identity map on $A$. 
\end{enumerate}
In the literature, such maps are called {\it involutions of the first kind} in contrast to {\it involutions of the second kind} which satisfy only b) and c). Since we make no use here of involutions of the second kind, we use the term `involution' from now on in place of the more cumbersome `involution of the first kind.'

Attached to $\theta$ is a sign $\ve(\ta) = \pm 1 \in F$ as discussed in the next two subsections.  
In the final subsection, we prove the conjugacy statement $(\star)$ from the introduction. 

\begin{Rmk} 
There is a natural dichotomy -- orthogonal versus symplectic -- for involutions $\ta$ as above 
(see \cite[2.1]{KMRT} and the surrounding discussion).  
For ${\rm char}\,F \neq 2$, we have $\ve(\ta) = 1$ (resp. $-1$) if  $\ta$ is orthogonal
(resp.~symplectic). 
In the case ${\rm char}\,F = 2$, the dichotomy is irrelevant to our purposes.  
\end{Rmk}

\subsection{The Split Case} \label{split-case}
We look first at the case of a matrix algebra $A = {\rm M}_n(F)$. 
As above, we write $\tp a$ for the transpose of any  $a \in A$. 
For any involution $\ta$ on $A$,  
the composition $\theta \circ \top$ is an $F$-algebra automorphism of $A$. As such 
automorphisms are  inner, there is a $b \in A^\times$ which is unique up to multiplication by an element of $F^\times$ such that 
\[
       \ut (\tp a) = b a b^{-1}, \hskip10pt \forall \, a \in A.
       \]
Equivalently,        
\begin{equation}   \label{form-of-theta}
  \ut a = b (\tp a) b^{-1}, \hskip10pt \forall \, a \in A.
\end{equation}
Using $\theta^2 = 1_A$, it follows that 
\[
  a  =  b (\tp b^{-1}) \, a \, (\tp b) b^{-1},  \hskip10pt \forall \, a \in A, 
  \]
 and  so $(\tp b) b^{-1}$ is a scalar matrix. That is, 
  \[
  \tp b = \ve \,b, \hskip10pt \text{for some $\ve \in F^\times$.}
  \]
Taking the transpose of each side, we obtain $b = \ve^2 \,b$. Thus $\ve = \pm 1$, so 
the matrix $b$ is symmetric or skew-symmetric.  We put $\ve(\ta) = \ve$.

The sign $\ve(\ta)$ can be expressed in terms of the eigenspaces of $\theta$ on $A$. 
We write $A^{\ta, +}$ and $A^{\ta, -}$ for the $+1$ and $-1$ eigenspaces of $\theta$ (resp.), so 
that $A =  A^{\ta, +} \oplus A^{\ta, -}$ if  ${\rm char}\,F \neq 2$. 
Using (\ref{form-of-theta}), one checks readily that 
\[
     \ut a = \pm a \hskip5pt \Longleftrightarrow \hskip5pt \tp(ab) = \pm \ve \, ab.
     \]
For ${\rm char}\,F \neq 2$, it follows that $\ve = \ve(\ta)$ can be characterized by 
\begin{equation} \label{e-space-formula}
\dim_F A^{\theta, +}  = \begin{cases} \dfrac{n(n+1)}{2} \hskip15pt  &\ve = 1 \\
                                        \dfrac{n(n-1)}{2}            &\ve = -1.
\end{cases}
\end{equation}
The formula still holds when ${\rm char}\,F = 2$ in the sense that the $F$-subspace of $\theta$-fixed vectors also has dimension $ \dfrac{n(n+1)}{2}$ in this case.

\subsection{}  \label{general-setting}
We return to the general setting. Thus 
$A$ is a central simple $F$-algebra and $\ta$ is an involution on $A$. 

As in the introduction, we write $\fac$ for an algebraic closure of $F$ and set $\Ac = A \otimes_F \fac$. 
Then $\Ac$ is a central simple $\fac$-algebra and hence $\Ac \cong {\rm M}_n(\fac)$ where
$n^2 = \dim_F A$. 
The map $\ta$ extends to an involution $\Bar{\ta}$ on $\Ac$ and so there is a sign 
$ \ve(\Bar{\ta})$ as in \S\ref{split-case}. More  pedantically, we set      
      \[
      \ve(\Bar{\ta}) = \ve(\alpha \Bar{\ta} \alpha^{-1}) 
\]
for any $\fac$-algebra isomorphism $\alpha: \Ac \to  {\rm M}_n(\fac)$. By definition, 
\[
      \ve(\ta) = \ve(\Bar{\ta}). 
      \]
Again, we write $A^{\ta, +}$ and $A^{\ta, -}$ for  the $+1$ and $-1$ eigenspaces (resp.) of 
$\theta:A \to A$ and use a parallel notation for $\Ac$. The maps  
\[
a \otimes \la \mapsto \la a :A^{\ta, \pm}  \otimes_F \fac  \overset{\simeq}{\longrightarrow}
   \Ac^{\Bar{\ta}, \pm}
\]         
are then isomorphisms of $\fac$-vector spaces. 
In particular, $\dim_F A^{\ta, +} = \dim_{\fac}  \Ac^{\Bar{\ta}, +}$. 
It follows that formula (\ref{e-space-formula})
also holds in this setting and characterizes the sign $\ve(\ta)$. 
(Again we just use the first line ($\ve = 1$) when ${\rm char}\,F = 2$.) 

\subsection{}  \label{conjugacy} 
We can now state and prove our conjugacy result. 

\begin{Thm}
Let $F$ be a field and let $A$ be a central simple $F$-algebra with involution $\ta$. 
We set $\ve = \ve(\ta)$. For any $a \in A$, there is a $g \in A^\times$ such that 
\[
          g a g^{-1} = \ut a \hskip5pt {\rm and} \hskip5pt \ut g = \ve \, g.
          \]
\end{Thm}

\begin{proof}
We look first at the split case $A = {\rm M}_n(F)$. By \cite[page 76]{Kap}, 
there is an $h \in A^\times$ such that
\[
  h a h^{-1} = \tp a\hskip5pt {\rm and} \hskip5pt \tp h = h.
\]
With $b$ as in (\ref{form-of-theta}), we have 
\begin{align*} 
   (bh) a (bh)^{-1} &=   b (h a h^{-1}) b^{-1} \\ 
                            &= b (\tp a) b^{-1} \\
                                                  &= \ut a.
                                                  \end{align*} 
Moreover, 
\begin{align*} 
      \ut (bh) &= b \, \tp h \, \tp b \, b^{-1} \\
                     &= \ve  \, bh \hskip15pt \text{(using   $\tp h  = h$ and $\tp b = \ve \, b$)}.
                     \end{align*} 
This establishes the result in the split case. 

Suppose now that $A$ is non-split. In particular, the field $F$ must be infinite. 
We use the notation introduced in \S \ref{general-setting}. Thus 
$\Ac = A \otimes_F \fac \cong {\rm M}_n(\fac)$  for $n^2 = \dim_F A$. 
We fix $a \in A$. Since the result holds in the split case, there is a 
$g \in \Ac^\times$ such that 
\begin{equation}  \label{solution}
   g a g^{-1} = \ut a\hskip5pt {\rm and} \hskip5pt \utt g = \ve \,g.
   \end{equation} 

Write $\ms$ for the set of $x$  in $A$ such that 
\[
   x a   = \ut x a, \,\, \ut x = \ve \, x.
   \]
Then $\ms$ is an $F$-subspace of $A$ and 
 $\msb = \ms \otimes_F \fac$ consists of all $x \in \Ac$ such that 
\[
   x a   = \utt x a, \,\, \utt x = \ve \, x.
   \]
Note that $\msb$ contains the invertible element $g$ of (\ref{solution}). 
To complete the proof, we show that $\ms$ contains an invertible element. 
We do so by borrowing an argument from Raghuram (see the proof of  \cite[Lemma 3.1]{Ragh}) and Tupan (see the proof of \cite[Lemma 2]{Tupan}). 

We write $\nr: A \to F$ and $\nrb: \Ac \to \fac$ 
for the reduced norm maps on $A$ and $\Ac$ (resp.).
We have 
\begin{equation} \label{norms}
\nrb(a \otimes 1) = \nr (a), \hskip10pt a \in A. 
\end{equation}
Let $s_1, \ldots, s_m$ be a basis of  $\ms$ and consider the polynomial 
$f \in F[X_1, \ldots, X_m]$ such that 
\[
f(\la_1, \ldots, \la_m) = \nr(\lambda_1 s_1 + \cdots + \lambda_m s_m)
\]
for $\la_1, \ldots, \la_m \in F$.
Given $\Bar{\la}_1, \ldots, \Bar{\la}_m \in \fac$, it follows from (\ref{norms}) that  
\[
  \nrb( s_1 \otimes \Bar{\la}_1 +  \cdots  +  s_m \otimes \Bar{\la}_m ) =  
  f( \Bar{\la}_1, \ldots,  \Bar{\la}_m).
  \]
By (\ref{solution}), $\msb$ contains the invertible element $g$, so that 
$\nrb(g) \neq 0$. In particular, $f$ is not the zero polynomial. 
As $F$ is infinite, it follows that there exist
$\la_1, \ldots, \la_m \in F$ such that $f(\la_1, \ldots, \la_m) \neq 0$. That is, 
\[
\nr(g') \neq 0 \hskip5pt \text{for $g' = \la_1 s_1 + \cdots + \la_m s_m \in \ms$}.
\]
Thus $g'$ is invertible and we have completed the proof.
\end{proof}

\section{An Application} \label{application} 
Let $F$ be a non-Archimedean local field and let $D$ be a quaternion division algebra over $F$. 
We write $\gamma$ for the canonical conjugation on $D$. 
 It is the unique symplectic involution on $D$ 
(\cite[2.21]{KMRT}). In particular, $\ve(\gamma) = -1$.
Given $a = (a_{ij}) \in  {\rm M}_n(D)$, we set ${}^{\gamma}a = ({}^{\gamma}a_{ij})$ and 
$\ut a = \tp {}^{\gamma}a$ (so that $\ut a$ has $ij$ entry ${}^{\gamma}a_{ji}$.) 
The resulting map $\ta$ defines an involution on ${\rm M}_n(D)$. We have $\ve(\ta) = -1$
(by a direct calculation or by \cite[2.20]{KMRT}). 
 Finally, we define an involutary automorphism $\iota$ of ${\rm GL}_n(D)$ by 
 $\uia g = \ut g^{-1}, \, g \in {\rm GL}_n(D)$. 

For $\pi$ an irreducible smooth (complex) representation of $G$, we write $\pic$ for the smooth dual or contragredient of $\pi$. It is well known that $\pi \circ \iota \cong \pic$ for any such $\pi$. 
In the terminology of \cite{RV-CMB, RV-JLT}, 
$\ia$ is a {\it dualizing involution}. This was proved by Mui\'{c} and Savin  
 in the case ${\rm char}\,F = 0$ \cite{MS} using character theory and by Raghuram  in all characteristics \cite{Ragh}. Raghuram's proof is an  adaptation of   a
geometric method used by  Gelfand and Kazhdan to show that $a \mapsto \tp a^{-1}$ is a dualizing involution on ${\rm GL}_n(F)$ \cite{GK}.   As noted in the introduction, 
Tupan found a completely elementary proof of Gelfand-Kazhdan's result using 
a) the classical observation that a square matrix over a field is conjugate to its transpose via a symmetric matrix and b) some $p$-adic topology  \cite{Tupan}. 
Our object in this section is to show that Tupan's method carries over to the group ${\rm GL}_n(D)$ using Theorem \ref{conjugacy} in place of a).

\subsection{}  \label{hypotheses}
It is convenient to use the axiomatic version of Tupan's method from  \cite{RV-JLT}. Thus let $G$ be the group of $F$-points of a reductive algebraic group over $F$ and write $\lag$ for the Lie algebra of $G$. Let $\vartheta:G \to G$ be an involutary anti-isomorphism on $G$ (induced by a corresponding map on the underlying algebraic group). We also write $\vartheta$ for the induced map on $\lag$.  As usual, for $x \in G$, we write ${\rm Int} (x)$ 
for the automorphism of $G$ given by conjugation by $x$ and 
${\rm Ad} (x)$ for the induced map on $\lag$. 

Let $\fo$ denote the valuation ring of $F$ and fix a uniformizer $\varpi$ in $F$.  Consider the following
hypotheses. 

\begin{enumerate}[{\bf (1)}]
\item
There is an  $\fo$-lattice $\ml \subset \lag$ and a map $c: \lag_1 \to G$ for a certain subset $\lag_1$ of $\lag$ such that the following hold.
\begin{enumerate}[{\rm (a)}]
\vskip3pt 
\item
${}^{\vartheta} \lag_1 = \lag_1$ and $\vartheta  \circ c = c \circ \vartheta$.
\vskip3pt 
\item
${\rm Ad} (x)  \lag_1 = \lag_1$  and 
${\rm Int} (x)  c(X)  = c ({\rm Ad} (x) X)$ for all $x \in G$ and $X \in \lag_1$.
\vskip3pt 
\item
${}^{\vartheta}  \ml = \ml$ and $\varpi \ml \subset \lag_1$.
\vskip3pt 
\item
For each $k \geq 1$, the restriction $c \mid \varpi^k \ml$ is a homeomorphism onto a compact open subgroup of $G$. In particular,  
the family $\{c (\varpi^k \ml) \}_{k \geq 1}$ consists of compact open subgroups and forms a neighborhood basis of the identity in $G$. 
\end{enumerate}

\vskip5pt

\item
For each $a \in G$, there is a $g \in G$ with ${}^{\vartheta}g = g$ such that 
$gag^{-1} = {}^{\vartheta}a$.
\end{enumerate}

\vskip5pt
Let   $\uia g = {}^{\vartheta}  g^{-1}, \, g \in G$. Subject to these hypotheses, \cite[Theorem 2.2]{RV-JLT} shows that the resulting map $\ia:G \to G$ is a dualizing involution.

\subsection{}
We apply the framework of \S \ref{hypotheses} to $G = {\rm GL}_n(D)$. 
We have $\lag = {\rm M}_n(D)$ and ${\rm Ad}(x) X = x X x^{-1}$ for $x \in G$ and $X \in \lag$. 
We write ${\rm Nrd}:{\rm M}_n(D) \to F$ for the reduced norm map and
$\od$ for the the unique maximal $\fo$-order in $D$. 

Suppose first that ${\rm char}\,F = 2$. We take  
$\lag_1 = \{ X \in {\rm M}_n(D) : {\rm Nrd}(1+X) \neq 0 \}$ and define $c:\lag_1 \to G$ by $c(X) = 1+X$. 
We set $\ml =  {\rm M}_n(\od)$. With $\vartheta = \ta$, it is then immediate that (a)-(d) of (1) hold. 
Hypothesis (2) holds by Theorem \ref{conjugacy}.  Thus $\ui g = \ut g^{-1}, \,g \in G$, 
defines a dualizing involution. 
 
Suppose now that ${\rm char}\,F \neq 2$. Choose any $y \in D^\times$ with ${}^{\gamma} y = -y$. 
For simplicity, we also write $y$ for the matrix $y I_n$ in $G$. For $a \in {\rm M}_n(D)$, we set
${}^{\vartheta} a = y \, \ut a \, y^{-1}$.  Since ${}^{\gamma} y = -y$, the resulting 
map $\vartheta$ is an involution.  By a direct calculation, 
\[
      {}^{\vartheta} a =   a \hskip5pt \Longleftrightarrow  \hskip5pt  \ut (ay) =  -  ay.
      \]
It follows that $\ve(\vartheta) = 1$.  We set 
\[
 \ml =   {\rm M}_n(\od)  \cap y \, {\rm M}_n(\od) \, y^{-1}.
\]
Note that ${}^{\gamma}(y^2) = y^2$, so that $y^2 \in F$. Thus 
\begin{equation}  \label{y-stable}
   y \,\ml \,y^{-1} =  y\, {\rm M}_n(\od) \,y^{-1}  \cap y^2 \, {\rm M}_n(\od) \,y^{-2} =  \ml.
   \end{equation}
Similarly, 
\begin{align}  \label{theta-stable} 
\ut \ml     &=  \ut {\rm M}_n(\od)  \cap \ut y^{-1} \, \ut {\rm M}_n(\od) \, \ut y   \\
&=  {\rm M}_n(\od)  \cap  y^{-1} \, {\rm M}_n(\od) \, y  \hskip10pt \text{(using $\ut y = -y$)} \notag \\
&= {\rm M}_n(\od)  \cap  y \,{\rm M}_n(\od) \, y^{-1}    \hskip10pt \text{(as $y^2 \in F$)} \notag \\
                &=  \ml.    \notag
                \end{align} 
By (\ref{y-stable}) and (\ref{theta-stable}), we have ${}^{\vartheta} \ml= \ml$. 
Again we take  $\lag_1 = \{ X \in {\rm M}_n(D) : {\rm Nrd}(1+X) \neq 0 \}$ and 
define $c:\lag_1 \to G$ by $c(X) = 1+X$. 
Then (a)-(d) of (1) hold. Hypothesis~(2) holds once more by Theorem \ref{conjugacy}. 
Hence $g \mapsto {}^{\vartheta} g^{-1}$ defines a dualizing involution of $G$. 
As $\vartheta$ and $\theta$ differ by an inner automorphism, it follows that
$\uia g = \ut  g^{-1}, \, g \in G,$  is also a dualizing involution.

\section{Revisiting the matrix algebra case} 
Let $n$ be a positive integer and let $a \in{\rm  M}_n(F)$. We set $V = F^n$, viewed as a set of column vectors, and write $V_a$ for the $F[X]$-module structure on $V$ given by $f(X)v = f(a) v$ for $f(X) \in F[X]$ and $v \in V$. 
We wish to prove the following.

\begin{Thm}
Given $a\in {\rm M}_n(F)$, there is a $g \in {\rm GL}_n(F)$ such that 
\[
(1) \hskip3pt   g a g^{-1} = \tp a \hskip5pt \text{and} \hskip5pt  (2) \hskip3pt \tp g = g. 
\]
Moreover, if $V_a$ is a cyclic module then any $g  \in {\rm GL}_n(F)$ that satisfies (1) also satisfies (2).  
If $V_a$ is non-cyclic then there is a  $g  \in {\rm GL}_n(F)$ that satisfies (1) but not (2).  
In other words, every $g \in {\rm GL}_n(F)$ that satisfies (1) also satisfies (2) 
if and only if the minimal and characteristic polynomials of $a$ coincide. 
\end{Thm}   
The result is not new -- see  \cite{TTZ}. Our goal is to provide a conceptual proof that 
relies principally on standard facts about finitely generated torsion modules over PIDs
and makes minimal use of special matrix calculations. 

\begin{proof}
The argument is spread over the next several subsections. 

\subsection{Reduction to Cyclic Case.} \label{reduce-to-cyclic}
There is an $h \in {\rm GL}_n(F)$ such that 
\[
h a h^{-1} = \begin{bmatrix} 
                           a_1    & 0     &  \cdots     &   \cdots  &   0  \\
                            0       & a_2  &0    &       &  \vdots \\
                            \vdots    &  0   & \ddots    &\ddots   &\vdots  \\
                            \vdots &  \vdots        &  \ddots     &        \ddots      & 0 \\
                             0 &   0       &  \cdots     &          0    & a_r 
                             \end{bmatrix} 
\]
with $a_i \in {\rm M}_{n_i}(F)$ and each $(F^{n_i})_{a_i}$ cyclic ($i=1, \ldots, r$). 
Suppose $b_i \in {\rm GL}_{n_i}(F)$ satisfies 
\[
b_i a_i b_i^{-1} = \tp a_i \hskip5pt \text{and} \hskip5pt  \tp b_i = b_i \hskip15pt (i=1, \ldots, r). 
\]
With 
\[
b =  \begin{bmatrix} 
                           b_1    & 0     &  \cdots     &   \cdots  &   0  \\
                            0       & b_2  &0    &       &  \vdots \\
                            \vdots    &  0   & \ddots    &\ddots   &\vdots  \\
                            \vdots &  \vdots        &  \ddots     &        \ddots      & 0 \\
                             0 &   0       &  \cdots     &          0    & b_r 
                             \end{bmatrix},  
\]
we then have $\tp b = b$ and 
\[
    b (h a h^{-1} ) b^{-1} =  \tp (h a h^{-1}).
    \]
Rearranging gives 
\[
      (\tp h b h) a (\tp h b h)^{-1}    =  \tp a
      \]    
with $\tp   (\tp h b h)  =  \tp h b h$.

\subsection{}
We recall some generalities about $F[X]$-modules that we will  apply eventually to $V_a$.  

For any $F[X]$-module $M$, we set $M^\vee = {\rm Hom}_F(M, F)$. We  write 
$\vform$ for the canonical (evaluation) pairing between $M$ and $M^\vee$:
\[
     \langle m, m^\vee \rangle =  m^\vee(m), \hskip15pt   m \in M,\,  m^\vee \in M^\vee.
     \]
 The space $M^\vee$ is an $F[X]$-module via     
\[
\langle m, f(X) m^\vee \rangle =  \langle f(X) m, m^\vee \rangle,  \hskip15pt   m \in M, 
 \,m^\vee \in M^\vee.
\]

For $m \in M$, let $\epsilon(m) = \langle m, - \rangle$, so that 
$\epsilon(m) \in M^{\vee \vee}$. The resulting map 
\begin{equation} \label{d-dual} 
   \epsilon = \epsilon_M: M \to M^{\vee \vee}
   \end{equation} 
is an $F[X]$-module homomorphism. Thus, if $M$ has finite dimension over $F$,  then (\ref{d-dual}) 
is an isomorphism of $F[X]$-modules. 

Suppose now that $M$ is a torsion $F[X]$-module that is also finite dimensional as an $F$-vector space. 
Let $\mathcal{P}$ denote the set of monic irreducible polynomials in $F[X]$. For $p \in \mcp$, we write
$M[p]$ for the $p$-primary component of $M$: 
\[
M[p] = \{ m \in M :  p^e  m = 0 \,\,\text{for some positive integer $e$} \}.
\]
Then 
\begin{equation} \label{p-primary}
     M = \bigoplus_{p \in \mcp} M[p].
     \end{equation} 
Moreover, for any $p \in \mcp$, there is a canonical isomorphism
\[
      M^\vee [p] \, \cong \, M[p]^\vee, 
      \]
so that 
\[ 
  M^\vee \cong \bigoplus_{p \in \mcp} M[p]^\vee, 
  \] 
which also follows directly from (\ref{p-primary}). 
By the Chinese Remainder Theorem, $M$ is cyclic if and only if each of its $p$-primary components
$M[p]$ is cyclic.  Further, for each $p \in \mcp$, $M[p]$ is cyclic if and only if $M[p]$ is indecomposable. 
Observe next that $M[p]$ is indecomposable if and only if $M[p]^\vee$ is indecomposable
(for any $p \in \mcp$). 
Indeed, if $M[p]$ splits as a non-trivial direct sum, then the  same holds for $M[p]^\vee$. For the other direction, note that 
\[
M[p]^{\vee \vee} \cong M^{\vee \vee} [p] \cong M[p].           
\]
Thus if $M[p]^\vee$ is a non-trivial direct sum then $M[p]$ splits in the same way. 
It follows that $M$ is cyclic if and only if $M^\vee$ is cyclic. 

Write ${\rm ann} \, M$ for the annihilator of the $F[X]$-module $M$:
\[
{\rm ann} \, M = \{ f(X) \in F[X] :  f(X)m = 0 \,\, \text{for all $m \in M$} \}.
\]
Note that 
\begin{equation} \label{ann-dual}
       {\rm ann} \, M =  {\rm ann} \, M^\vee.
       \end{equation} 
Indeed, 
\begin{align*} 
      f(X) \in {\rm ann}\, M \,\, &\Longleftrightarrow \,\, \langle f(X)m, m^\vee \rangle = 0,  \,\,\, 
              \forall \, m \in M, \, m^\vee \in M^\vee \\ &\Longleftrightarrow \,\, 
\langle m, f(X) m^\vee \rangle = 0,  \,\,\, \forall \, m \in M, \, m^\vee \in M^\vee \\ 
   &\Longleftrightarrow  \,\, f(X) \in {\rm ann} \, M^\vee.
   \end{align*} 
Thus, for $M$ cyclic,  
\begin{align*}  
M  &\cong F[X] / {\rm ann} \,M \\ 
&= F[X] / {\rm ann} \, M^\vee  \hskip10pt (\text{by (\ref{ann-dual})}) \\
&\cong M^\vee. 
\end{align*}                   

In general, $M$ is a direct sum of cyclic submodules. Since each of these cyclic summands is 
self-dual, we see again that  $M \cong M^\vee$. Taking $M = V_a$, we have an isomorphism of 
$F[X]$-modules
\begin{equation} \label{eta}
\eta:V_a \overset{\simeq}{\longrightarrow} V_a^\vee.
   \end{equation} 

Let $\hform$ denote the usual dot product on $V$, that is, 
\[
    ( v, w) =  \tp v  w, \hskip15pt v, w \in V.
    \]
Then, for any $b \in {\rm M}_n(F)$ and $v, w \in V$,  
\[
            (  b v, w) = ( v, \tp b w). 
            \]
In particular, 
\[
            ( \tp a v, w) = ( v,  a w),      \hskip15pt v, w \in V.
            \]
Thus, if we set $\gamma(v) = (v, -)$ for $v \in V$, then 
 \[
 \gamma: V_{\tp a} \overset{\simeq}\longrightarrow   \vad
 \]
 is  an isomorphism of $F[X]$-modules. Hence
 \[
 \gamma^{-1} \circ \eta: V_a \overset{\simeq}{\longrightarrow}  V_{\tp a} 
 \]
 is an isomorphism of $F[X]$-modules. 
 This means there is a $g \in {\rm GL}_n(F)$ such that 
 \begin{equation} \label{trans-conj}
     g a g^{-1} = \tp a.
     \end{equation} 
     
 \subsection{}   \label{cyclic-case}
We write  $\tp \eta: V_a^{\vee \vee} \to \vad$ for the dual isomorphism to (\ref{eta}) and set 
 $\eta^\vee = \tp \eta \circ \epsilon$ with $\epsilon$ as in (\ref{d-dual}) (for $M = V_a$). 
Thus  $\eta^\vee: V_a \to V_a^\vee$ is again an isomorphism of $F[X]$-modules.  Unwinding the definitions, one checks that it is characterized by the identity 
\begin{equation} \label{eta-check} 
    \langle     v, \eta^\vee (w)  \rangle     =       \langle  w, \eta(v)  \rangle,\hskip15pt 
  \forall \,\,    v , w \in V.
\end{equation} 

We wish to show that if $V_a$ is cyclic then any $g$ that satisfies (\ref{trans-conj})
 is necessarily symmetric. The crux of our argument is the following. 
\begin{Lem}  \label{self-dual-map} 
Assume that $V_a$ is cyclic. Then, with notation as above, $\eta^\vee = \eta$. 
\end{Lem} 

\begin{proof} 
Let $v_1$ be a generator of $V_a$. 
In (\ref{eta-check}), we can write $v = f(X) v_1$ and $w = g(X) v_1$ for suitable 
$f(X), \,g(X) \in F[X]$. 
Then 
\begin{align*} 
    \langle  w, \eta (v) \rangle       &=  \langle g(X)v_1, \eta (f(X) v_1) \rangle \\
                                                 &=   \langle f(X) g(X) v_1, \eta (v_1) \rangle.
                                                 \end{align*} 
In the same way, 
\[
     \langle  v, \eta^\vee (w) \rangle   =    \langle f(X) g(X) v_1, \eta^\vee (v_1) \rangle.  
\]
Hence 
\[
        \langle g(X) v_1, \eta( f(X) v_1) \rangle = \langle g(X) v_1, \eta^\vee(f(X) v_1 ) \rangle.
        \]
Thus 
        \[
        \langle w, \eta (v) \rangle = \langle w, \eta^\vee ( v) \rangle, \hskip15pt \forall \,\, v, w \in V, 
\]
and so $\eta = \eta^\vee$ as claimed. 
\end{proof} 
The matrix $g$ of (\ref{trans-conj}) satisfies 
\[
    g v    =   \gamma^{-1} (\eta v),
    \]
    or $ \gamma ( g v) = \eta (v)$, for  $v \in V$.  Hence
\[ 
(gv, w)    =   \langle w, \eta( v) \rangle,  \hskip15pt \forall \,\, v, w \in V. 
\] 
Therefore
\begin{align*} 
        (v, gw)    &= (gw, v)   \\   
        &=   \langle v, \eta( w) \rangle,      \hskip15pt \forall \,\, v, w \in V.   \notag 
        \end{align*} 
Thus, for all $v, w \in V$,  
\[
 (gv, w) = (v, gw) \,\, \Longleftrightarrow \,\, \langle w, \eta (v) \rangle = \langle  v, \eta (w) \rangle.
    \]
Using (\ref{eta-check}), it follows that 
\begin{equation}    \label{equivalence}
      \tp g = g  \,\,  \Longleftrightarrow \,\, \eta^\vee = \eta.
      \end{equation} 
Hence, by Lemma \ref{cyclic-case}, $\tp g = g$ whenever $V_a$ is cyclic.           
\subsection{}
Suppose now that $V_a$ is non-cyclic. 
It remains to show that there is a 
$g \in {\rm GL}_n(F)$ such that $g a g^{-1} = \tp a$ with $\tp g \neq g$. By (\ref{equivalence}), 
it suffices to show that $\eta^\vee \neq \eta$ with $\eta$ as in (\ref{eta}).

As $V_a$ is non-cyclic, we can write 
\[
V_a = V_1 \oplus V_2 \oplus \cdots \oplus V_r
\]
where a) each $V_i$ is a cyclic submodule
and b) ${\rm Hom}_{F[X]}(V_1, V_2^\vee) \neq \{ 0 \}$. For example, as noted above, some $p$-primary
component $V_a[p]$ must be non-cyclic, so this component splits as a non-trivial direct sum $V_1 \oplus V_2$. As $F[X]/(p)$ is the unique composition factor of $V_1$ and $V_2^\vee$, condition
b) surely holds. 

We fix a non-zero $F[X]$-module homomorphism $\eta_{12}:V_1 \to V_2^\vee$. 
Dualizing gives a non-zero $F[X]$-module map $\tp \eta_{12}:V_2^{\vee \vee} \to V_1^\vee$. 
Then $\eta_{12}^\vee  =  \eta_{12} \circ \epsilon_{V_2}$ is a non-zero $F[X]$-module map from
$V_2$ to $V_1^\vee$. 
We also 
fix $F[X]$-module isomorphisms $\eta_i:V_i \to V_i^\vee$ (for $i = 1, \ldots, r$). 
Using block matrix notation, we set 
\[
\eta = \begin{bmatrix} 
                 \eta_1    & \eta_{12}     &  0     &   \cdots  &   0  \\
                            0       & \eta_2  &0    &       &  \vdots \\
                            \vdots    &  0   & \ddots    &\ddots   &\vdots  \\
                            \vdots &  \vdots        &  \ddots     &        \ddots      & 0 \\
                             0 &   0       &  \cdots     &          0    & \eta_r 
                             \end{bmatrix},  
\]
so that $\eta:V_a \to V_a^\vee$ is an isomorphism of $F[X]$-modules. 
We have 
\[
\eta^\vee = \begin{bmatrix} 
                 \eta_1    &   0   &  0     &   \cdots  &   0  \\
                 \eta_{12}^\vee   & \eta_2  &0    &       &  \vdots \\
                    0    &  0   & \ddots    &\ddots   &\vdots  \\
                    \vdots &  \vdots        &  \ddots     &        \ddots      & 0 \\
                  0 &   0       &  \cdots     &          0    & \eta_r 
                             \end{bmatrix}.  
\]
In particular, $\eta^\vee \neq \eta$. This completes the proof.    
\end{proof}

\end{document}